\documentclass[12pt]{article}
\usepackage[utf8]{inputenc}
\usepackage[english]{babel}
\usepackage[margin=0.75in]{geometry}
\usepackage{color}
\usepackage{enumerate}
\usepackage{amsmath, amsthm, amssymb, amsfonts}
\usepackage{mathtools}
\usepackage[all,cmtip]{xy}
\usepackage{todonotes}
\usepackage{hyperref}
\usepackage{todonotes}
\allowdisplaybreaks
\newtheorem{theorem}{Theorem}
\newtheorem*{corollary}{Corollary}
\newtheorem{lemma}{Lemma}

\newtheorem*{claim*}{Claim}
\newtheorem*{prob*}{Problem}

\newtheorem{case'}{Case}
\newtheorem{case''}{Case}

\newtheorem*{conjecture*}{Conjecture}
\newtheorem*{note}{Note}

\newcommand{\F}{\mathbb{F}}

\newenvironment{theorem*}[2][Theorem]{\begin{trivlist}
\item[\hskip \labelsep {\bfseries #1}\hskip \labelsep {\bfseries #2}]}{\end{trivlist}}
\newenvironment{lemma*}[2][Lemma]{\begin{trivlist}
\item[\hskip \labelsep {\bfseries #1}\hskip \labelsep {\bfseries #2}]}{\end{trivlist}}
\newenvironment{corollary*}[2][Corollary]{\begin{trivlist}
\item[\hskip \labelsep {\bfseries #1}\hskip \labelsep {\bfseries #2}]}{\end{trivlist}}
\begin{document}
\title{English Translation of Tschebotar{\"o}w's ``The Problem of Resolvents and Critical Manifolds''}
\author{Hannah Knight\\
This translation was supported in part by NSF Grant DMS-1811846.} 
\date{}
\maketitle
\begin{center}  \end{center}

\begin{center}\section*{Bibliographic Information}\end{center}

\begin{itemize}

\item Author: N. G. Tschebotar{\"o}w
\item Title: The problem of resolvents and critical manifolds
\item Newspaper: Izvestia USSR Academy of Sciences
\item Series: Mathematical Series
\item Date: 1943
\item Volume: 7
\item Issue: 3
\item Pages: 123-146

\end{itemize}

\newpage

\begin{center} \section*{The Problem of Resolvent and Critical Manifolds}\end{center}

\begin{abstract} This paper is devoted to the study of the problem of resolvents, i.e., the problem of finding the resolvent by a given equation, whose coefficients depend on several independent parameters, the number of parameters in the coefficients being as small as possible.  The author connects the problem to the study of critical manifolds in the space of equation parameters. \end{abstract}

This article is devoted to the problem of resolvents posed by F. Klein  \cite{Kl3} and significantly advanced by D. Hilbert \cite{Hi1} and \cite{Hi2}. It consists in finding a rational transformation of an algebraic equation containing variable parameters such that the the transformed equation contains as few independent parameters as possible. While Klein required that the conversion coefficients be either rational or containing predetermined irrationalities, Hilbert assumed that they depend on the roots of the auxiliary equations, which in turn admit resolvents with a small number of parameters.

The methods by which attempts were made to solve the problems of Klein and Hilbert are also significantly different. Klein's problem was reduced to the problem giving the Galois group of a given equation a Lie group structure\footnote{The literal translation was ``dress the Galois group of a given equation by a Lie group''}, represented in a space of the smallest possible dimension. I \cite{Ts5} showed that an equation of degree $n$ with unboundedly variable coefficients and an alternating Galois group cannot be transformed into a resolvent depending on less than $n - 3$ parameters. Thus it was found that the solution Klein's problem is significantly different from the solution to Hilbert's problem.

For Hilbert's Problem, Hilbert himself proposed a special trick, giving for $n = 5, 6, 7, 8, 9$ the values of the number of parameters in resolvents given in the following table:

\begin{center}

\begin{tabular}{c | c c c c c}
n & 5 & 6 & 7 & 8 & 9\\
\hline
s & 1 & 2 & 3 & 4 & 4
\end{tabular}
.
\end{center}

Wiman \cite{Wi7} showed that for $n \geq 9$ there exist resolvents with $s \leq n - 5$ parameters. Moreover, it remains unclear whether these values of s are the smallest possible.

\marginpar{Page 124}

In this article, a new method for studying resolvents is introduced. This method is based on the construction of higher critical manifolds in the spaces whose coordinates are the parameters on which the coefficients of a given equation depend. A higher critical manifold, I define to be a set of points in the space of parameters at which several roots of the equation coincide. For each equation, all types of higher critical manifolds can be defined. If for an equation, we can find a chain of critical manifolds of length $s$ such that each critical manifold is contained in the previous one, then during a birational transformation of the equation, this chain is mapped into the same chain with each manifold having a smaller dimension than the previous one. This shows that the equation cannot have a resolvent with less than $s$ parameters.  This gives a lower bound for the value of $s$ (while Hilbert and Wiman give an upper bound) if we restrict ourselves to rational resolvents.\footnote{This restriction is essential dimension vs. resolvent degree.} The study of irrational resolvents, which requires a detailed study of critical manifolds for relative fields, I postpone until the next article. \footnote{I don't think he ever wrote this later article.}

Applying this method to equations with unboundedly variable coefficients and an alternating Galois group gives a lower bound of $$s = \left [ \frac{n-1}{2} \right ],$$ which is very close to the values found by Hilbert

\begin{center}

\begin{tabular}{c | c c c c c}
n & 5 & 6 & 7 & 8 & 9\\
\hline
s & 2 & 2 & 3 & 3 & 4
\end{tabular}
.
\end{center}

However, a comparison of both $s$ for $n = 5$ shows that the use of irrational resolvents is significant.

The question of whether the lower bound found for $s$ will also be the upper bound is connected to the question of the actual construction of the resolution with the help of critical manifolds. Currently, I have some considerations that make it very likely a positive answer to this question. However, they are not enough to make any definition conclusion. \footnote{I don't think he ever wrote anything more on these ``considerations.''}

\begin{center}\section{Higher critical manifolds}\end{center}
Given an equation
\begin{equation}\label{eq1} f(x) = x^n + a_1x^{n-1} + \cdots + a_n = 0 \end{equation}
whose coefficients are polynomials in a number of independent variables with coefficients in the field of complex numbers. These variables themselves will be represented as complex coordinates of points \marginpar{Page 125} spaces $U$, that we assume are $m$-dimensional in complex sense, or $2m$-dimensional in the real sense. If for any point $P$ of the space $U$, the discriminant $D$ of equation (\ref{eq1}) does not vanish, then the roots
$$x_1, x_2, \cdots, x_n$$
 of equation (\ref{eq1}) are defined in a neighborhood of the point $P$ as analytic functions of $m$ variables $u_i$, which can be extended to the whole space $U$. However, such an extension is not always well-defined. Furthermore, if the equation (\ref{eq1}) is irreducible in the field of rational functions of $u_1,u_2,\cdots, u_m$, then in $U$ one can find a closed path such that if any of the roots extends along it $x_1, x_2, \cdots, x_n$ are permuted. In fact, if in the continuation of all closed paths a root $x_i$ is only sent to 
$$x_1, x_2, \cdots, x_k (k<n),$$
then all elementary symmetric functions of these roots will be well defined in the whole space $U$. Hence it would follow that they are rational with respect to each of the variables $u_1, u_2, \cdots, u_m$ individually and therefore are rational functions in terms of the $u_i$. Thus, equation (\ref{eq1}) would have a factor of degree $k < n$, with coefficients rationally depending on $u_1, u_2 \cdots, u_m$.

The set of permutations obtained as a result of extending the roots along various closed paths of the space $U$ is called the group $G$ of equation (\ref{eq1}). It is easy to make sure that it coincides with the Galois group of the equation, if we take the domain of rationality to be the set of rational function of $u_1, u_2, \cdots u_m$ with complex coefficients. 

Any closed path in the space $U$ can be contracted to a point. It follows that in $U$ there exist infinitesimal closed paths, through which the roots undergo non-trivial permutations, and the union of the latter group is the whole group $G$.\footnote{The literal translation was ``The totality of the latter has composite the whole group $G$.''} It is clear that the roots moving on the infinitesimal closed path must be infinitely close, so that such paths lie in the a neighborhood of the manifold points 
\begin{equation}\label{eq2} D(u_1,u_2,\cdots,u_m) = 0,\end{equation}
where $D$ is the discriminant of equation (\ref{eq1}).

To determine the manifolds of points in the neighborhood of which there are closed paths along which the roots undergo permutations of a give cycle type (more precisely, a given conjugacy class of elements of the group $G$), consider the product
\begin{equation}\label{eq3} \prod_{S \in G} (t_1x_{a_1} + t_2x_{a_2} + \cdots + t_nx_{\alpha_n}) = \Phi(t_1,t_2,\cdots,t_n),\end{equation}
\marginpar{Page 126}
where $t_1,t_2,\cdots,t_n$ are independent variables and 
$$S = \begin{pmatrix} 1 & 2 & \cdots & n\\
\alpha_1 & \alpha_2 & \cdots & \alpha_n\end{pmatrix},$$
and the product is over all permutations $S$ of the group $G$. The coefficients in (\ref{eq3}) obviously are rational functions in the variables $u_1, u_2, \cdots, u_m$.

We write the permutation $S$ in cycle notation
$$S = (1,2,\cdots, \mu_1)(\mu_1+1, \cdots, \mu_2) \cdots (\mu_{k-1} + 1, \cdots, n).$$
We say that a point $P$ of the space $U$ lies in the manifold corresponding to the permutation $S$ if the coordinates coincide in $P$:
\begin{equation}\label{eq4} \begin{rcases} x_1 = x_2 = \cdots = x_{\mu_1}\\
x_{\mu_1+1} = \cdots = x_{\mu_2}\\
\cdots \cdots \cdots \cdots \cdots \cdots\\
x_{\mu_{k-1}+1} = \cdots = x_n.\end{rcases}\end{equation}

If the manifold corresponding to permutation $T$ (we will denote it by $U_T$), is contained in $U_S$
$$U_T \subset U_S,$$ then we say the permutation $T$ is higher than the permutation $S$:
\begin{equation}\label{eq5} T > S \end{equation}
This will occur if and only if each set of roots in the cycles for the permutation $S$ are contained in the cycles for $T$.

Furthermore, by $U_{(S)}$ we denote the manifold of points of $U$ corresponding to one class of permutations containing $S$. We say that a class $(T)$ is above $(S)$ if $(T)$ contains a permutation which is higher than any permutation from $(S)$.

We will use $\Phi_S$ to denote the element in the product in (\ref{eq3}) corresponding to $S$, in which we subject the variable $t_1, t_2, \cdots t_n$ to the following conditions:
\begin{equation}\label{eq6} \begin{rcases} t_1 + t_2 + \cdots + t_{\mu_1} &=0,\\
t_{\mu_1+1} + \cdots + t_{\mu_2} &= 0,\\
\cdots \cdots \cdots \cdots \cdots \cdots\\
t_{\mu_{k-1}+1} + \cdots + t_n &= 0.\end{rcases} \end{equation}

Then the following holds:

\begin{theorem} In order for a point point $P$ to lie in $U_{(S)}$, it is necessary and sufficient that for this point all the coefficients of $\Phi_{S}$ are zero.\end{theorem}

\begin{proof}[Proof 1] The condition is necessary. In fact, if there are the coinciding points in (\ref{eq4}), the linear form $$t_1x_1 + t_2x_2 + \cdots + t_nx_n$$
\marginpar{Page 127} in the conditions from (\ref{eq6}) vanishes since

\begin{equation}\label{eq7} \begin{rcases}
&t_1x_1 + t_2x_2 + \cdots + t_nx_n =\\
&=t_1x_1 + \cdots + t_{\mu_1-1}x_{\mu_1-1} - (t_1 + \cdots + t_{\mu_1-1}x_{\mu_1} +\\
&+ t_{\mu_1+1}x_{\mu_1+1} + \cdots + t_{\mu_2-1}x_{\mu_2-1} - (t_{\mu_1+1} + \cdots + t_{\mu_2-1})x_{\mu_2}+\\
&\cdots \cdots \cdots \cdots \cdots \cdots \cdots \cdots \cdots \cdots \cdots \cdots \cdots \cdots \cdots \cdots \cdots \cdots \cdots \cdot \\
&t_{\mu_{n-1}+1}x_{\mu_{n-1}+1} + \cdots + t_{n-1}x_{n-1} - (t_{\mu_{n-1}+1} + \cdots + t_{n-1})x_n = \\
&= t_1(x_1 - x_{\mu_1}) + \cdots + t_{\mu_1-1}(x_{\mu_1-1} - x_{\mu_1}) + \\
&+ t_{\mu_1+1}(x_{\mu_1+1} - x_{\mu_2} + \cdots + t_{\mu_2-1}(x_{\mu_2-1} - x_{\mu_2}) +\\
&\cdots \cdots \cdots \cdots \cdots \cdots \cdots \cdots \cdots \cdots \cdots \cdots \cdots \cdots \cdots \cdots \cdots \cdots \cdots \cdot \\
&+ t_{\mu_{n-1}+1}(x_{\mu_{n-1}+1} - x_n) + \cdots + t_{n-1}(x_{n-1} - x_n). \end{rcases} \end{equation}

 If a point $P$ has coinciding coordinates corresponding to some other permutation in the same class\footnote{according to table (\ref{eq4}}), then some other factor of the expression will vanish in (\ref{eq3}). Thus,
$$\Phi_S = 0.$$
 A 2nd condition is sufficient. In fact, if the point $P$ satisfies 
$$\Phi_S = 0,$$
then assuming the conditions in (\ref{eq6}), at least one of the factors of the expression in (\ref{eq3}) is zero. Presenting it in a form similar to (\ref{eq7}), we will ensure that there are coinciding coordinates similar to (\ref{eq4}).  Namely, if the following factor vanishes

$$t_1s_{\alpha_1} + t_2x_{\alpha_2} + \cdots + t_nx_{\alpha_n} = 0,$$

where

$$T = \begin{pmatrix} 1 & 2 & \cdots & n\\
\alpha_1 & \alpha_2 & \cdots & \alpha_n \end{pmatrix} \in G,$$

then the table of equalities will become

\begin{center}\begin{align*}
&x_{\alpha_1} = x_{\alpha_2} = \cdots = x_{\alpha_{\mu_1}}\\
&x_{\alpha_{\mu_1+1}} = \cdots = x_{\alpha_{\mu_2}}\\
&\cdots \cdots \cdots \cdots \cdots \cdots\\
&x_{\alpha_{\mu_{k-1}+1}} = \cdots = x_{\alpha_n}.
\end{align*}\end{center}

That is, it will correspond to the permutation

\begin{center}\begin{align*}
&\begin{pmatrix} \alpha_1 & \alpha_2 & \cdots & \alpha_n\\
1 & 2 & \cdots & n \end{pmatrix} (1, 2, \cdots, \mu_1)(\mu_1+1, \cdots, \mu_2) \cdots\\
&\text{ } \cdots (\mu_{k-1} + 1, \cdots, n) \begin{pmatrix} 1 & 2 & \cdots & n\\
\alpha_1 & \alpha_2 & \cdots & \alpha_n\end{pmatrix} = T^{-1}ST,\end{align*}\end{center}
from which it follows that the point $P$ belongs to the manifold $U_{(S)}$ as desired.

\end{proof}

In the ring of polynomials of $u_1, \cdots, u_m,$ rationally expressed in terms of the coefficients in equation (\ref{eq1}), to each manifold $U_{(S)}$, there corresponds an ideal of polynomials given by the polynomials \marginpar{Page 128} which vanish on the manifold $U_{(S)}$.\footnote{The literal translation was ``there corresponds polynomial ideal representing the collection of polynomials vanishing on the manifold.''} Let this ideal be denoted by $A_{(S)}$. To determine whether a given polynomial
$$h(u_1, u_2, \cdots, u_m),$$ 
expressed in terms of the coefficients in equation (\ref{eq1}) is in $A_{(S)}$, we write $h$ in terms of the roots $x_1, x_2, \cdots, x_n$ of equation (\ref{eq1}) and set some of them equal according to the table of equalities in (\ref{eq4}). We then see if the polynomial vanishes after that.

If the variables $u_1, u_2, \cdots, u_m$ are included in the coefficients of equation (\ref{eq1}) linearly, the ideal will be prime. In fact, without loss of generality we can assume that then certain variables $u_1, u_2, \cdots, u_m$ are rationally (or even linearly) expressed in terms of the coefficients $a_1, a_2, \cdots, a_n$. Then if the product $g \cdots h$ belongs to $A_{(S)}$, we express $g$ and $h$ and the variables $u_1, u_2, \cdots, u_n$ in terms of $x_1, x_2, \cdots, x_n$ and set some of them equal according to the table of equalities in (\ref{eq4}). If in this case the product $g \cdots h$ vanishes, then at least one of its factors must vanish. This proves that the ideal $A_{(S)}$ is prime. 

Note that if, instead of the permutation $S$, we take any other permutation $T^{-1}ST$ in the same conjugacy class ($T \in G$), then we get the same ideal. This follows from the fact that the table of equalities for the permutation $T^{-1}ST$ is obtained from the table in (\ref{eq4}) by applying the permutation $T$. On the other hand, every rational function of $x_1, x_2, \cdots, x_n$ rationally expressed in terms of $u_1, u_2, \cdots, u_n$, does not change if the permutation $T$ is applied to it.

\begin{center}\section{Inertia Groups}\end{center}
 
The set of permutations of the roots of equation (\ref{eq1}) formed by all infinitesimal paths taken in the neighborhood of some point $P$ of the space $U$, we will call the inertia group of $P$, by analogy with algebraic theory (or rather, p-adic numbers). We would like to define an inertia group for a point $P$ in the manifold $U_{(S)}$. On one hand, it is clear that roots which coincide when following along a path near the point $P$ should coincide at the point $P$. It follows that any permutation of the inertia is not higher than one of the permutations of the conjugacy class $(S)$.  The converse is true only with certain conditions. For ease of study, first assume that the variables $u_1, u_2, \cdots, u_m$ are included linearly in the coefficients $a_1, a_2, \cdots, a_n$. In particular, this will occur if we let $U$ be the space of coefficients $a_1, a_2, \cdots, a_n$.

In this case, equation (\ref{eq1}) can we rewritten in the form
\begin{center}\begin{equation}\label{eq8} f_0(x) + f_1(x)(u_1-p_1) + f_2(x)(u_2-p_2) + \cdots + f_m(x)(u_m-p_m) = 0,\end{equation}\end{center}

\marginpar{Page 129} 
where $p_1, p_2, \cdots, p_m$ are the coordinates of the point $P$.  The polynomials $f_0(x), f_1(x), \cdots f_m(x)$ do not have a common factor, since the roots of all their common factors would be the roots of equation (\ref{eq1}) and at the same time would not depend on $u_1, u_2, \cdots, u_m$, which contradicts the irreducibility of equation (\ref{eq1}).

Suppose that the point $P$ is in $U_{(S)}$, but is not in any of the higher critical manifolds. Then the polynomial $f_0(x)$ will have constant roots whose equalities exactly correspond to table (\ref{eq4}). Denoting them by $b_1, b_2, \cdots, b_n$, we will thus have

\begin{equation}\label{eq9} \begin{rcases} b_1 = b_2 = \cdots = b_{\mu_1}\\
b_{\mu_1+1} = \cdots = b_{\mu_2}\\
\cdots \cdots \cdots \cdots \cdots \cdots\\
b_{\mu_{k-1}+1} = \cdots = b_n.\end{rcases}\end{equation}

where none of the values

$$b_1, b_{\mu_1+1}, \cdots, b_{\mu_{k-1}+1}$$

are equal to each other. Since the polynomials $f_0(x), f_1(x), \cdots, f_m(x)$ are relatively prime, we can choose constants $c_1, c_2, \cdots, c_m$ so that the polynomial

$$g(x) = c_1f_1(x) + c_2f_2(x) + \cdots + c_mf_m(x)$$

is relatively prime to $f_0(x)$, i.e. it does not vanish for any of the values from (\ref{eq9}). Let us produce a linear transformation of $u_1, u_2, \cdots, u_m$ such that

\begin{center}\begin{align*}
u_1 - p_1 &= c_1v_1\\
u_2 - p_2 &= c_2v_1 + v_2\\
\cdots \cdots \cdots \cdots \cdots \cdots\\
u_m - p_m &= c_mv_1 + v_m
\end{align*}\end{center}

Then equation (\ref{eq8}) can be rewritten as follows
\begin{center}\begin{equation}\label{eq10} f_0(x) + g(x)v_1 + f_2(x)v_2 + \cdots + f_m(x)v_m = 0.\end{equation}\end{center}

The variables $v_1, v_2, \cdots, v_m$ vanish at point $P$.  Assuming that 
\begin{center}\begin{equation}\label{eq11} v_2 = v_3 = \cdots = v_m = 0\end{equation}\end{center} 
we get
\begin{center}\begin{equation}\label{eq12} v_1 = -\frac{f_0(x)}{g(x)} = - \frac{(x-b_1)^{\mu_1}(x-b_2)^{\mu_2} \cdots (x - b_{\mu_{k-1}+1})^{\mu_k}}{g(x)}, \end{equation}\end{center}

where the denominator $g(x)$ is coprime to the numerator, i.e it does not evaluate to zero at $v_1 = 0$.  If in this equation we set 
\begin{center}\begin{equation}\label{eq13} v_1 = \rho e^{i\theta} \end{equation}\end{center}
where $\rho > 0$ is a very small number, and $\theta$ will take values between $0$ and $2\pi$. Then from the theorem on the continuity of the roots of algebraic equations, \marginpar{Page 130} it follows that $\mu_1$ from the roots of equation (\ref{eq12}) will be very close to $b_1$, $\mu_2$ to $b_2$, etc. In this case, as $\theta$ runs from $0$ to $2\pi$, each of the roots of these categories will be mapped into each other, and thus all the roots will undergo a permutation of the same cycle type as the permutation $S$.

At the same time, equalities (\ref{eq12}) and (\ref{eq13}) define a small circle around the point $P$ in the space $U$. Thus, we have proved the following theorem:

\begin{theorem}\label{Thm2} The inertia group of the point $P$ contains a permutation whose cycles correspond to the rows in the table of equalities of the roots for the point $P$. \end{theorem}

We study what types of all permutations are in the inertia group of the point $P$.  If at a point $P$ there is a small intersection with critical manifolds lower than $U_{(S)}$, for example the manifold $U_{(S_1)}$, then any neighborhood of the point $P$ contains an infinite number of points for which $U_{(S_1)}$ is the highest of the critical manifolds on which they lie. This implies, by Theorem \ref{Thm2}, that in every neighborhood of the point $P$, there are ways in which the roots $x_1, x_2, \cdots, x_n$ undergo a permutation similar to the permutation $S_1$. Hence we have:

\begin{lemma}\label{Lem1} The inertia group of the point $P$ contains permutations of all cycle types corresponding to the table of equalities for the critical manifolds on which $P$ lies. \end{lemma}

The converse is also true. For the proof we need the following lemmas:

\begin{lemma}\label{Lem2} In each chain of critical manifolds

$$U_{(S_1)} \supset U_{(S_2)} \supset \cdots \supset U_{(S_r)},$$
 
the dimension of each subsequent manifolds is at least one less than the dimension of the previous one. \end{lemma}

\begin{proof} This follows from the fact that, as we saw in Section 1, the ideal corresponding to each critical manifold is a prime ideal. In fact there is a theorem [see van der Waerden \cite{Vd6}, p. 63], which states that two prime ideals, with one containing the other, only have the same dimension if the coincide. \end{proof}

\begin{corollary} The lowest critical manifold (that is, not contained in any other) has a complex dimension $\leq m-1$ \end{corollary}

We now consider the space $U$ as a $2m$-dimensional real space.  From what has just been proved it follows that the lower critical manifolds have dimension $\leq 2m-2$. Let $U$ be a point space belonging to one of the higher critical manifolds $U_{(S)}$ (perhaps not to the highest critical manifold), but not belonging to any of the other higher  \marginpar{Page 131} critical manifolds, that has a neighborhood containing the closed curve $C$, such that when following the curve $C$, the roots $x_1, x_2, \cdots, x_n$ undergo the permutation $S$. We will follow $C$ on a two-dimensional manifold $A$ (we stretch the film), and let it not intersect either $P$ or with any of the critical varieties which are higher than the lowest critical variety. This can always be achieved by small deformations of the manifold $A$, since, by virtue of Lemma \ref{Lem2}, every higher critical manifold has a real dimension $\leq 2m-4$. On the other hand, if $S$ is not the identity permutation, then $A$ necessarily intersects the lower critical manifolds. In fact, otherwise, dividing $A$ into arbitrarily small areas, we could consider the boundary of each of these areas as a closed curve located in a neighborhood of a non-critical point. Therefore, when going around each of these contours, the roots undergo the identity permutation. But since going along the path $C$ is equivalent to following these boundaries in succession, then going along $C$ would also make the same permutation on the roots, which contradicts the assumption.

We can assume that the manifold $A$ is algebraic. Then, using a small deformation, the number of intersections with the critical manifolds can be made finite.  Let these points of intersection be $P_1, P_2, \cdots, P_k$ and let closed loops around them (lying on $A$) correspond to lower permutations $S_1, S_2, \cdots, S_k.$ Obviously, they can be numbered in such a way that 
$$S_1 \cdots S_2 \cdot \cdots \cdot S_k = S.$$

This implies

\begin{lemma}\label{Lem3} All permutations of the inertia group of the point $P$ are the product of lower permutations corresponding to the lowest critical manifold in a neighborhood of the point $P$. \end{lemma}
 
From the proof of this lemma, it follows that
\begin{lemma}\label{Lem4} The  dimension of the lowest critical manifold is exactly equal to $2m-2$. \footnote{The literal translation says ``eigenvalue'' rather than ``dimension''.}\end{lemma}

It is significant that if this were less than $2m-2$, then the manifold $A$, after a small deformation, would not have intersections with the critical manifolds.

In general, now we can refine Lemma \ref{Lem2}. The manifold $U_{(S_2)}$ is the intersection of either two different lower manifolds or two components \footnote{The literal translation is ``fields.''} of the same manifold $U_{(S_1)}$. Indeed, by Lemma \ref{Lem3}, the higher corresponding permutation is a product of lower permutations, and permutations corresponding to the two lower manifolds at the intersection form a higher critical \marginpar{Page 132} manifold of dimension $2m-4$. Continuing the argument, we obtain the following lemma:

\begin{lemma}\label{Lem5} In a chain of critical manifolds
$$U_{(S_1)} \supset U_{(S_2)} \supset \cdots \supset U_{(S_k)},$$
the real dimension of each is respectively $2m - 2, 2m-4, \cdots, 2m-2k.$\end{lemma}

Assume that the inertia group of the point $P$ contains the permutation $S$, maybe not the highest for the point $P$. It follows from Lemma \ref{Lem3} that $S$ can represented as a product of lower permutations 
$$S = S_1 \cdot S_2 \cdot \cdots \cdot S_k.$$
This means that the point $P$ lies in the intersection of the lower critical manifolds. By Lemma \ref{Lem5}, their intersection forms a critical manifold of dimension $2m-2k$. Therefore, in a neighborhood of the point $P$there lie $\infty^{2m-2k}$ points of the manifold $U_{(S)}$, of which only at most $\infty^{2m-2k-2}$\footnote{Note by Tschebotar\"{o}w: When it comes to algebraic varieties, the use of the old notation $\infty^\nu$ gives the presentation great clarity and also cannot lead to misunderstanding} points lie to the highest critical manifold. Thus, in a neighborhood of the point $P$ there are points for which $U_{(S)}$ is the highest variety on which they lie. Comparing with Theorem \ref{Thm2}, we get the following theorem:

\begin{theorem}\label{Thm3} The inertia group of a point $P$ contains a permutation $S$ if and only if there are points in its neighborhood for with $U_{(S)}$ is the highest critical manifold. \end{theorem}

We proceed to consider more general cases. Let the coefficients $a_1, a_2, \cdots, a_n$ in equation (\ref{eq1}) be arbitrary polynomials of $u_1, u_2, \cdots, u_m$, which we will consider as independent variables. For such a space, we can find both critical manifolds and inertia groups. Each point in the space $U$ corresponds to one point in the space $A$, and a closed path in the space $U$ corresponds to a closed path in the space $A$. Conversely, to each point in the space $A$ there are several points in the space $U$, and some of them can be multiple, and thus a closed path in the space $A$ can correspond to an open path in the space $U$.  In other words, several once-repeated closed paths in the space $A$ can fit in a closed path in $U$. 

A point in $U$ is critical only if it corresponds to a critical point in $A$. However, it can happen that a critical point of the space $A$ corresponds to ordinary points of the space $U$. This is the case if in order to produce a closed path in a neighborhood of some point \marginpar{Page 133} space $U$, we must go around one of the corresponding points of the space $A$ the number of times equal to the product of the orders of each of the permutations. Thus the number of critical manifolds contained in each other in the space $U$ does not exceed the number in the space $A$. Leaving a detailed analysis of the relationship between the spaces $U$ and $A$ until another case, I note that in the space $U$ non-prime polynomial ideals can correspond to critical manifolds.

If the parameters $u_1, u_2, \cdots, u_m$ in the coefficients of equation (\ref{eq1}) are subject to the algebraic equation

\begin{center} \begin{equation}\label{eq14} g(u_1, u_2, \cdots, u_m) = 0, \end{equation} \end{center}
we must select an algebraic surface in the space $U$ which is the algebraic surface defined by equation (\ref{eq14}), which we will consider as the parametric space $U_1$ of equation (\ref{eq1}). In the space $U_1$, only those closed paths whose points lie on the surface defined by (\ref{eq14}) will be considered, and the inertia groups of critical points will be determined only for such paths.  We note that in such spaces $U_1$, the monodromy theorem does not hold: there may exist a function on $U_1$ that is unique with respect to all infinitesimal closed paths, but is not unique in general. The classical (two-dimensional) theory of Riemannian surfaces gives examples of such functions.

In some cases, the variables $u_1, u_2, \cdots, u_m$ are independent but at the same time the Galois group $G$ of equation (\ref{eq1}) is given. Then the parametric space of equation (\ref{eq1}) is constructed as follows. We compose a function of the roots of equation (\ref{eq1}), which belongs to the group $G$. We denote it by $u_{m+1}$. Assume that in the field of rational functions of $u_1, u_2, \cdots, u_m$, it satisfies the irreducible equation

\begin{center} \begin{equation}\label{eq15} g(u_1,u_2,\cdots,u_m, u_{m+1}) = 0.\end{equation} \end{center}

We define the parametric space $U$ as the surface (\ref{eq15}) in $(m+1)$-dimensional space with coordinates $u_1, u_2, \cdots, u_m, u_{m+1}$.

Note that if we define a closed path in the space $U'$ with coordinates $u_1, u_2, \cdots, u_m$ (the projection of the space $U$) then in the space $U$ it will be closed only if, when it goes around the function $u_{m+1}$ it returns to its original value; in other words if the permutation $S$ made by the roots of equation (\ref{eq1}) when going around this path is contained in $G$. Otherwise, the closed path in $U$ can be represented as a $k$-fold path in $U'$, where $k$ is the smallest exponent for which 
$$S^k \subset G.$$

\begin{center} \section{The problem of resolvents in various formulations} \end{center}

\marginpar{Page 134}

Let two algebraic equations 
\begin{center} \begin{equation}\label{eq16} f(x) = 0, \end{equation}
\begin{equation}\label{eq17} g(y) = 0 \end{equation}\end{center}
 be given of the same degree $n$ with isomorphic Galois groups $G$ and $\overline{G}$.  We will consider the $p$-rationality domain formed by the union of the fields formed by the coefficients and also the functions \footnote{N.b. presumably he means the invariant functions} of the roots of the equation (\ref{eq16}), (\ref{eq17}) belonging to the group $G, \overline{G}$. We have

\begin{theorem}\label{Thm4} The roots of equation (\ref{eq16}), (\ref{eq17}) are in rational dependence given by

\begin{center}\begin{equation}\label{eq18} y_i = \alpha_0 + \alpha_1x_i + \cdots + \alpha_{n-1}x_i^{n-1} \text{  } (i = 1,2, \cdots, n), \end{equation} \end{center}

if and only if the equation 

\begin{center}\begin{equation}\label{eq19} F(u) = 0, \end{equation}\end{center} \footnote{He does not say what $F$ is, but he says what its roots are, which is sufficient to construct $F$.}

whose roots are given by 

\begin{center}\begin{equation}\label{eq20} u = \sum_{k=1}^{n-1} t_k(x_1^ky_z + x_2^ky_x + \cdots + x_n^ky_n), \end{equation} \end{center}

has at least one rational root, where  $t_1, t_2, \cdots, t_n$ are undefined variables (which we will include into the field of rationality). \end{theorem}

\begin{proof} If there are dependencies (\ref{eq18}), where $\alpha_0, \alpha_1, \cdots, \alpha_{n-1}$ are rational quantities, then substituting (\ref{eq18}) into (\ref{eq20}), we obtain for each of the quantities

\begin{center}\begin{equation}\label{eq21} u_k = x_1^ky_1^k + x_2^ky_2 + \cdots + x_n^ky_n \text{  } (k = 0, 1, 2, \cdots, n-1) \end{equation} \end{center} 
the expression

\begin{center}\begin{equation}\label{eq22} u_k = \alpha_0s_k + \alpha_1s_{k+1} + \cdots + \alpha_{n-1}s_{n+k-1}, \end{equation}\end{center}
where $s_m$ is sum of the $m$-th powers of the roots of equation (\ref{eq16}), i.e. rational values of formula (\ref{eq22}). It follows that in this case $u_k$, and hence $u$ as well, are rational values.

Now suppose that the conditions of the theorem are satisfied. In order to make equation (\ref{eq19}), with root $u$, let us use $S$ and $\overline{S}$ to denote the formally identical permutations made respectively over the roots $x_1, x_2, \cdots, x_n$ and $y_1, y_2, \cdots, y_n$.  Producing expression (\ref{eq20}) and permutation $S$, or equivalently permutation $\overline{S}^{-1}$ [since the expression (\ref{eq20}) is invariant under permutations $S\overline{S}$], we will send $u$ to $u^S$. Forcing $S$ to be in the group $G$, we obtain quantities, of which the elementary symmetric functions \marginpar{Page 135} are symmetric with respect to each of the systems $x_1, x_2, \cdots, x_n$ and $y_1, y_2, \cdots, y_n$, whereby the quantities
$$u^S = \sum_k t_ku_k^S$$
are the roots of equation (\ref{eq19}), whose coefficients are rational. 

The Galois group of the field formed by all roots $x_i, y_i$ is a subgroup of the direct product $G \times \overline{G}$. In this field, it is either the group of composed of products of $S\overline{S}$ (we denote this group by $G\overline{G}$) or to its supergroup. However, if equations (\ref{eq16}) and (\ref{eq17}) do not have multiple roots, then there are not permutations other than $S\overline{S}$ that would leave $u$ invariant, in other words, all $u_k$. Indeed, suppose there is such a permutation, $S_1\overline{S_2} (S_1 \neq S_2)$. Multiplying it by $(S_1\overline{S_1})^{-1}$, we get the permutation
$$\overline{S_2}\overline{S_1}^{-1} = \overline{S_3} \neq 1,$$
under which all $u_k$ will be invariant. Assuming that 
$$\overline{S_3} = \begin{pmatrix} 1 & 2 & \cdots & n\\
\alpha_1 & \alpha_2 & \cdots & \alpha_n\end{pmatrix}$$
we get the equalities 
\begin{center} $u_k - u_k^{\overline{S_3}} = x_1^k(y_1 - y_{\alpha_1}) + x_2^k(y_2 - y_{\alpha_2}) + \cdots + x_n^k(y_n - y_{\alpha_n}) = 0.$\\
$(k = 0, 1, \cdots, n-1).$\end{center}

Considering them as a system of homogeneous linear equations of $y_k - y_{\alpha_k}$ with a non-zero determinant [Vandermonde determinant of the roots of equation (\ref{eq16})] we obtain

$$y_k - y_{\alpha_k} = 0 \text{ } (k = 1,2, \cdots, n),$$
which contradicts the fact that equation (\ref{eq17}) has no multiple roots.

Hence it follows that 

$$u = \sum_k t_k u_k$$

belongs to the groups $G\overline{G}$. If equation (\ref{eq19}) has a rational root, then this means that Galois group  of the field formed by the roots $x_1, x_2, \cdots, x_n$ and $y_1, y_2, \cdots, y_n$ is one of the groups associated with $G\overline{G}$. By changing the numbering of the roots $y_1, y_2, \cdots, y_n$, we can make it equal to $G\overline{G}$. Then all the quantities in $u_k$ will be rational. Solving with respect to $\alpha_0, \alpha_1, \cdots \alpha_{n-1}$ in the system of equations in (\ref{eq22}), we get rational expressions for $\alpha_0, \alpha_1, \cdots, \alpha_{n-1}$. Then formulas (\ref{eq18}) will satisfy the requirements of the theorem as desired. 

\end{proof}

\begin{note} \marginpar{Page 136} If equation (\ref{eq17}) does not have multiple roots, then formulas (\ref{eq18}) are invertible. This follows, for example form Theorem 47 of my ``Foundations of Galois Theory.'' (\cite{Ts4}) \end{note}

We will equations (\ref{eq16}) and (\ref{eq17}) \textbf{transformable equations} if there is a rational relationship between their roots of the type in (\ref{eq18}).

The resolvent problem posed in various formulations by Klein (\cite{Kl3}) and Hilbert (\cite{Hi1}, \cite{Hi2}) can be posed in the following form which is more general and at the same time more easily formulated: An equation is given whose coefficients depend on $m (\leq n)$ parameters. We want to find an equation that can be converted into it whose coefficients depend on the smallest possible number of parameters (which we denote by $s$).

The coefficients (or parameters) of equation (\ref{eq16}) and (\ref{eq17}) may not depend rationally on each other. The task is to establish algebraic dependencies between them so that two conditions are met:

\begin{enumerate}[1)] 

\item Transformation (\ref{eq18}) must be invertible. We have seen that this condition is always satisfied if equations (\ref{eq16}) and (\ref{eq17}) do not have multiple roots.

\item Equation \ref{eq19} must have a root which depends rationally on the coefficients of equations (\ref{eq16}) and (\ref{eq17}), as well as on functions of their root belonging to the groups $G$ and $\overline{G}$ respectively.

\end{enumerate}

Compare the above formulation of the problem of resolvents with the formulations proposed by Klein and Hilbert.

\begin{prob*}[Klein's Problem] Using equation (\ref{eq16}), find an equation (\ref{eq17}) that depends on the smallest possible number $s$ of parameters so that the dependencies in (\ref{eq18}) between their roots are such that the coefficients $\alpha_0, \alpha_1, \cdots, \alpha_{n-1}$ rationally depending on the coefficients of equation (\ref{eq16}). \footnote{This is equal to essential dimension.}  \end{prob*} 

As a generalization of Klein's problem, we allow of the dependence of $\alpha_0, \alpha_1, \cdots, \alpha_{n-1}$ on some pre-set irrationalities.

Klein's problem is more particular in that it imposes more requirements than Hilbert's problem.

\begin{prob*}[Hilbert's Problem] Using equation (\ref{eq16}), find an equation (\ref{eq17}) (we will call it the resolvent), depending on the smallest number ($s$) of parameters, so that the coefficients $\alpha_0, \alpha_1, \cdots, \alpha_{n-1}$ of the dependencies in (\ref{eq18}) between their roots were determined using the auxiliary equation, which, in turn, should have a resolvent with $\leq s$ parameters, and the dependence coefficients between the roots of the latter should again depend on the roots of the equation admitting a resolvent with $\leq s$ parameters, and so on. The number of auxiliary equations obtained in this case should be finite. \footnote{This is equal to resolvent degree} \end{prob*}

Let us prove that Hilbert's problem is a special case of the problem we formulated. It suffices to consider the case \marginpar{Page 137} when $G$ is a simple group. In fact, if $G$ has a normal divisor $H$, then we must first solve Hilbert's problem for an auxiliary equation whose Galois group is isomorphic to the quotient group $G/H$. Then, considering this equation auxiliary and attaching its roots to the domain of rationality, we lower the Galois group of the given equation to $H$. Thus the number $s$ in the Hilbert problem for equations with Galois group $G$ is equal to the largest of the numbers $s$ in the Hilbert problem for equations with Galois groups $H$ and $G/H$.

Suppose that equation (\ref{eq16}) with a simple Galois group $G$ has a $s$-parametric equation (\ref{eq17}) in the Hilbert sense. Let us prove that (\ref{eq17}) will also be a resolvent in our sense. Suppose to the contrary: let equation (\ref{eq19}) have no rational root, where we consider the composite fields of the coefficients of equations (\ref{eq16}) and (\ref{eq17}) to be the rationality domain. In this rationality domain, the field formed by the roots of equations (\ref{eq16}) and (\ref{eq17}) has a Galois group which is a subgroup of the direct product $G \times \overline{G}$, where therefore is not contained in the group $G\overline{G}$ formed by permutations of type $S\overline{S}$, nor in any of the $G\overline{G}$ groups.

First of all, we prove that the group of equation (\ref{eq16}) does not decrease if the coefficients of equation (\ref{eq17}) are added to the rationality domain. Indeed, otherwise this field would contain a natural irrationality belonging to the real subgroup $G$. Its intersection with all conjugate subgroups is the identity group, by virtue of which this irrationality will be the root of the equation whose Galois group is isomorphic to $G$. Thus, the field of the roots of this equation should be regarded as an auxiliary (the field formed by the quantities $\alpha_0, \alpha_1, \cdots, \alpha_{n-1}$, should contain it), containing the starting roots of equation (\ref{eq16}). This means that the resolvent for the auxiliary equation will have no fewer parameters than the resolvent (\ref{eq17}), and the new auxiliary equation will again contain the same irrationality etc., so that the process never ends.  

Switching the roles of equations (\ref{eq16}) and (\ref{eq17}), we will prove that in the field of rationality formed by the coefficients of equations (\ref{eq16}) and (\ref{eq17}), the Galois group for both of these equations is isomorphic to $G$.

If the fields $K_1, K_2$, formed from the roots of equations (\ref{eq16}) and (\ref{eq17}) respectively, are coprime, that is, their intersection is a composite of coefficient fields, then the composite group $K_1 \times K_2$ is isomorphic to $G \times \overline{G}$. Since this group has no other normal divisors besides $G$ and $\overline{G}$, the group $G\overline{G}$ to which the root of equation (\ref{eq19}) belongs is not a normal divisor of the group $G \times \overline{G}$ \footnote{n.b. This is false without further assumptions, e.g. $G = \overline{G} = \F_p$}, and moreover, and its intersection with the conjugate subgroups is the identity group. Therefore, the group of equation (\ref{eq19}) is isomorphic to $G \times \overline{G}$. Hence it follows that the roots of equations (\ref{eq16}) and (\ref{eq17}) are contained in the field \marginpar{Page 138} roots of equation (\ref{eq19}), so that the latter, being an auxiliary equation, has no simpler solution than each of equations (\ref{eq16}) and (\ref{eq17}).

The intersection of the fields $K_1$ and $K_2$ is a normal field, and therefore, inside each of the fields $K_1, K_2$ is a normal divisor of $G, \overline{G}$ respectively. By the simplicity of the latter group, if the intersection is not a composite of coefficient fields, it must equal each of the fields $K_1, K_2$.

The existence of a rational dependence of type (\ref{eq18}) between the roots of equations (\ref{eq16}) and (\ref{eq17}) does not always follow from this. In particular, it does not hold if and only if $G$ admits several different (i.e. not changing each other by changing the order of the numbers) representations in the form of a transitive permutation group of $n$ numbers, that is, if $G$ has outer automorphisms. For example, this is the case when $G$ is a alternating group of $6$ numbers. However, in this case our statement remains valid. Namely, if equations (\ref{eq16}) and (\ref{eq17}) are such that the fields formed by their roots $K_1$ and $K_2$ give the same result, then we can transform the resolvent (\ref{eq17}) so that after the transformation, there are rational equations are established like (\ref{eq18}) between the roots of (\ref{eq16}) and (\ref{eq17}). For this purpose we need to construct an equation of degree $n$ whose root, being an element of the field $K_2$, would belong to the field $K_1$, which contains the roots of equation (\ref{eq16}). The equation constructed in this way, being convertible into equation (\ref{eq16}), at the same time remains $s$-parametric. 

The question of the conditions under which both problems are equivalent requires a special study.

\begin{center} \section{Conditions necessary for the existence of a resolvent} \end{center}

In this paper we restrict ourselves to the case when the region of rationality of the resolvent (\ref{eq17}) coincides with the rationality domain of equation (\ref{eq16}). Let the coefficients of equation \ref{eq16} depend rationally on the parameters $u_1, u_2, \cdots, u_m$, and let the $s$ parameters of the resolvent (\ref{eq17}), $v_1, v_2, \cdots, v_s$ be entire rational functions of $u_1, u_2, \cdots, u_m$. \footnote{The literal translation is ``Let the $s$ parameters of the resolvent (\ref{eq17}) $v_1, v_2, \cdots, v_s$ are given with them, $u_1, u_2, \cdots, u_m$ in such a correspondence, what are their entire rational functions.''}   Then to each closed path in the space 
$$U(u_1, u_2, \cdots, u_m)$$
will correspond a well-defined closed path of the space 
$$V(v_1,v_2, \cdots, v_s).$$
If in this case, the first of the paths is in a neighborhood of any one point (i.e. it is infinitely small), then the second path is also in a neighborhood of a certain point.

If around a closed path $C$ in a neighborhood of the point $P$, the roots of equation (\ref{eq16}) under a permutation $S$, then passing along \marginpar{Page 139} the corresponding path $\overline{C}$ of $V$, the roots of equation (\ref{eq17}) undergo a permutation $\overline{S}$, which differs from $S$ only by other permuted objects. In fact, let the rational root of equation (\ref{eq19}) have the form 
$$\sum_k t_ku_k,$$
where
$$u_k = x_1^ky_1 + x_2^ky_2 + \cdots + x_n^ky_n, \text{  } (k  = 0, 1, \cdots, n-1).$$

If when taking the path $C$, the roots $x_1, x_2, \cdots, x_n$ undergo the permutation

$$S = \begin{pmatrix} x_1 & x_2 & \cdots & x_n\\
x_{\alpha_1} & x_{\alpha_2} & \cdots & x_{\alpha_n}\end{pmatrix},$$

and when taking the path $\overline{C}$, the roots $y_1, y_2, \cdots, y_n$ undergo the permutation

$$\overline{S_1} = \begin{pmatrix} y_1 & y_2 & \cdots & y_n\\
y_{\beta_1} & y_{\beta_2} & \cdots & y_{\beta_n}\end{pmatrix},$$

then, since the function $u_k$ should be

$$x_{\alpha_1}^ky_{\beta_1} + x_{\alpha_2}^ku_{\beta_2} + \cdots + x_{\alpha_n}^ky_{\beta_n} = x_1^ky_1 + x_2^ky_2 + x_n^ky_n \text{ } \footnote{I think this might be a typo and there should be $+ \cdots +$ inserted here} (k = 0, 1, \cdots, n-1),$$
from which we get 
\begin{center}\begin{equation}\label{eq23} x_1^k(y_1 - y_{\gamma_1}) + x_2^k(y_2 - y_{\gamma_2} + \cdots + x_n^k(y_n - y_{\gamma_n}) = 0, \text{ }(k = 0, 1, \cdots, n-1),\end{equation}\end{center}
where
$$\begin{pmatrix} \alpha_1 & \alpha_2 & \cdots & \alpha_n\\
\beta_1 & \beta_2 & \cdots & \beta_n\end{pmatrix} = \begin{pmatrix} 1 & 2 & \cdots & n\\
\gamma_1 & \gamma_2 & \cdots & \gamma_n\end{pmatrix},$$
from which we get 
$$S_1 = \begin{pmatrix} 1 & 2 & \cdots & n\\
\beta_1 & \beta_2 & \cdots & \beta_n\end{pmatrix} = \begin{pmatrix} 1 & 2 & \cdots & n\\
\alpha_1 & \alpha_2 & \cdots & \alpha_n \end{pmatrix}\begin{pmatrix} \alpha_1 & \alpha_2 & \cdots & \alpha_n\\
\beta_1 & \beta_2 & \cdots & \beta_n\end{pmatrix} = S\begin{pmatrix} 1 & 2 & \cdots & n\\
\gamma_1 & \gamma_2 & \cdots & \gamma_n\end{pmatrix},$$
i.e. 
$$\begin{pmatrix} 1 & 2 & \cdots & n\\
\gamma_1 & \gamma_2 & \cdots & \gamma_n\end{pmatrix} = S^{-1} \cdot S_1.$$

But since the path $C$ extends non-critical points of the space $U$, the determinant of the system of homogeneous linear equation (\ref{eq23}) on the path $C$ does not vanish anywhere, and so we get 
\begin{center}\begin{equation}\label{eq24} y_1 = y_{\gamma_1}, \text{ } y_2 = y_{\gamma_2}, \cdots, y_n = y_{\gamma_n}.\end{equation}\end{center}

On the other hand, the dependencies (\ref{eq18}) between the roots of equations (\ref{eq16}) and (\ref{eq17}) are assumed to be invertible, from which we can conclude that the path from the roots of equation (\ref{eq17}) is also multiple.  Therefore, from the equations (\ref{eq24}) we have 

$$\gamma_1 = 1, \gamma_2 = 2, \cdots, \gamma_n = n,$$

and hence 
$$S_1 = S.$$

So we arrive at the theorem:

\begin{theorem}\label{Thm5} \marginpar{Page 140} If the parameters of equation (\ref{eq17}) are entire rational functions of the parameters of equation (\ref{eq16}) and the equations are transformable into each other, then the corresponding points of the spaces formed by the parameters of equation (\ref{eq17}) have inertia groups containing subgroups that are isomorphic to the inertia groups of the points in the space of parameters of equation (\ref{eq16}). \end{theorem}

Consider the case when the coefficients of equation (\ref{eq16}) are linear in terms of the parameters $u_1, u_2, \cdots, u_m$. Let equation (\ref{eq17}) be transformable into equation (\ref{eq16}), if we put polynomials $v_1, v_2, \cdots, v_s$ instead of the parameters ; $v_1 = \varphi(u_1, u_2, \cdots, u_m), \cdots, v_s = \varphi_s(u_1, u_2, \cdots, u_m).$  It follows from Theorem 5 that the critical manifolds of the space $U$ correspond to the points of the critical manifolds of the space $V$. If we agree to consider as admissible only the closed paths in the space $V$  corresponding to the closed paths in the space $U$, then by Theorem 5 the inertia groups of the corresponding points are isomorphic. 

Let the space $U$ contain a chain of critical manifolds 

$$U_{(S_1)} \supset U_{(S_2)} \supset \cdots \supset U_{(S_q)}.$$

Let there be correspond manifolds in the space $V$

$$V_1 \supset V_2 \supset \cdots \supset V_q.$$

Let us prove that these manifolds have different dimensions. On the one hand, no pair of neighboring manifolds $V_i, V_{i+1}$ can coincide. Indeed, let $P$ be a point of the space $U$ that belongs to the manifold $U_{(S_{i+1})}$ but does not belong to any higher manifold. By virtue of Theorem (\ref{Thm3}), in its neighborhood there are points $P'$ of the manifold $U_{S_i}$ not belonging to the manifold $U_{S_{i+1}}$, and therefore the groups of inertia do not coincide with the inertia group of the point $P$ (the order of the latter is higher). Suppose that in the space $V$, the points $P, P'$ correspond to points $Q, Q'$. Under our condition regarding admissible closed paths, it follows from Theorem \ref{Thm5} that the inertia groups of the points $Q$ and $Q'$ are different. Both of them lie on the manifold $V_i$, but at the same time the point $Q$ lies on the manifold $V_{i+1}$ but the point $Q'$ does not. 

On the other hand, the manifolds $V_i$ correspond to prime ideals $B_i$. In fact, to determine whether a polynomial $g(v_1, v_2, \cdots, v_s)$ is in the ideal $B_i$, it is necessary to express the variables $u_1, u_2, \cdots, u_m$ in terms of $u_1, u_2, \cdots, u_m$, which in turn are expressed in terms of of the roots $x_1, x_2, \cdots, x_n$ of equation (\ref{eq16}). After this, we must equate these roots to each other in accordance with the table of equalities corresponding to the permutation $S_i$. The polynomial $g$ lies in the ideal $B$ if and only if it vanishes after the indicated operations. From this criterion it follows that the product $g \cdot h$ lies in $B_i$. Thus $B_i$ is a prime ideal. 

From this \marginpar{Page 141} it follows that all the dimensions of $V_1, V_2, \cdots, V_q$ are different, and since the dimension of the manifold $V_{q}$ is not less than zero, $V_1$ has dimension $\geq q - 1$.  But since $V_1$ is a critical manifold and $V$ contains non-critical points, the dimension of the space $V$ is less than $q$, and so we arrive at the theorem:

\begin{theorem}\label{Thm6} If an equation of type (\ref{eq16}), with linear coefficients in terms of the parameters, contains a chain of $q$ critical manifolds, each contained in the previous one, then every rational resolvent contains at least $q$ parameters. \end{theorem}

This theorem makes it possible to establish a lower bound for the number of resolvent parameters.

As an example, we consider the most general equations (i.e. with independent variables as coefficients) with an alternating Galois group. In this case, $U$ is a surface defined in $(n+1)$-dimensional space of $a_1, a_2, \cdots, a_n, z$ satisfying the equation
$$z^2 - D(a_1,a_2, \cdots, a_n) = 0,$$
where $D$ is the discriminant of equation (\ref{eq16}). It is easy to verify that each type of even permutation in the spaces $U$ corresponds to a critical manifold. Indeed, we can always construct an equation (possible reducible) for which the Galois group consists of the degrees of any even permutation $S$. This equation will be a particular form of equation (\ref{eq16}), since the coefficients of the latter are independent variables. On the other hand, the equation constructed in this way will certainly have $U_s$ in its parametric space (which is obtained from $U$ by critical manifolds giving some of the coefficients particular values).

For example, if 
$$S = (1, 2, \cdots, \mu_1)(\mu_1 + 1, \cdots, \mu_2) \cdots (\mu_{k-1} + 1, \cdots, n),$$

then as such a particular equation we can take 
$$(x^{\mu_1} - w_1)(x^{\mu_2-\mu_1} - w_2) \cdots (x^{n-\mu_k-1} - w_k) = 0,$$
where $w_1, w_2, \cdots, w_k$ are independent variables.

Thus, in our case, we obtain a lower bound for the possible number of parameters of the resolvent of equation (\ref{eq16}), if we consider the maximum number of links in all possible chains of even permutations of the $n$th degree, in which the next term is higher than the previous one. One of these chains can be 

$$S_1 < S_2 < \cdots < S_{[\frac{n-1}{2}]},$$
where 
\begin{center} $S_1 = (123), S_2 = (12345), \cdots, S_i = (123 \cdots 2i+1);$\\
$i = 1, 2, \cdots, [\frac{n-1}{2}].$\end{center}

This chain consists of $[\frac{n-1}{2}]$ links.

Let \marginpar{Page 142} us prove that the alternating group of degree $n$ does not contain a chain of length greater that $[\frac{n-1}{2}]$. Indeed noting in each permutation the number of cycles contained in it, including monomial cycles (invariant numbers), we note that for even permutations this number has the same parity as $n$. On the other hand, if, for example,
$$T_1 < T_2,$$
then the permutation $T_2$ must certainly contain fewer cycles than $T_1$. Since they have the same parity, these numbers must differ by at least $2$. Therefore the number of cycles in the permuations of the chain

$$T_k > T_{k-1} > \cdots > T_2 > T_1$$

cannot be less than 

$$1, 3, 5, \cdots, 2k-1$$
respectively. 

But $T_1$ cannot be the identity or a transposition, thus 

$$2k - 1 \leq n -2,$$
hence 
$$k \leq [\frac{n-1}{2}].$$

Thus the desired lower bound for the number of parameters is equal to 
\begin{center}\begin{equation}\label{eq25} s = [\frac{n-1}{2}].\end{equation}\end{center}

We compare these values with those obtained by Hilbert (\cite{Hi2}) for $5 \leq n \leq 9$:

\begin{center} \begin{tabular}{c | c | c | c | c | c}
\text{ }  & 5 & 6 & 7 & 8 & 9\\
\hline
s [according to formula (\ref{eq25})] & 2 & 2 & 3 & 3 & 4\\
\hline
s [according to Hilbert] & 1 & 2 & 3 & 4 & 4
\end{tabular}\end{center}

This comparison shows that for $n = 5$ the introduction of irrational resolvents actually reduces the number of parameters. In the cases $n = 6, 7, 9$, the number $s$ determined by the different means, actually gives the same value for the number of parameters. Finally, the case $n = 8$ requires additional investigation.

In conclusion, \marginpar{Page 143} we mention two main issues that arise when solving the problem under consideration: 

\begin{enumerate}[1)]
\item \textbf{The question about irrational resolvents}. We have just seen that the introduction of irrationalities into the domain of coefficients of the equation can significantly reduce the number of parameters in the resolvent. These irrationalities must be introduced so that not all the critical manifolds remain irreducible. With this space, $U$ and $V$ will be in algebraic, but not rational, dependence. Let the coordinates $v_1, v_2, \cdots, v_s$ of the space $V$ be determined by the equations 
$$\varphi^i(u_1,u_2, \cdots, u_m, v_s) = 0 \text{ } (i = 1, 2, \cdots, s).$$

Viewing the area of rationality as the totality of rational functions of $u_1, u_2, \cdots, u_m$, we can find a primitive element

$$\xi = c_1v_1 + c_2v_2 + \cdots + c_sv_s$$

in the field formed by the the elements $u_1, \cdots, u_m; v_1, \cdots, v_s.$ Let 

\begin{center} \begin{equation}\label{eq26}F(u_1, u_2, \cdots, u_m, \xi) = 0.\end{equation}\end{center}

Let us  denote by $W$ the surface formed by the coordinates $u_1, u_2, \cdots, u_m, \xi$, which are subject to equation (\ref{eq26}). The coordinate spaces $U$ and $V$ are rationally expressed in terms of the coordinates of the space $W$. By the inertia groups of equations (\ref{eq16}), (\ref{eq17}) we mean the set of permutations of the roots which arise when following infinitesimal closed paths in the space $W$.  In the latter space, closed paths can correspond to simple or closed paths repeated several times in the spaces $U, V$ . The problem is to construct such a $W$ such that due to repetition of closed paths in $U$ and $V$, the inertia groups corresponding to different critical manifolds coincide.

\item \textbf{The question of the upper bound for the number of parameters for the resolution.} If the space $U$ of parameters of equation (\ref{eq16}) is given in which critical manifolds form chains of length $\leq s$, then the question arises of the possibility of constructing a $s$-parametric resolvent. If $U_s$ is the highest critical manifold of the space, then in the space $V$ it must correspond to a critical manifold of dimension zero. This indicates the path to the actual construction of resolvents. The issue requires further research.
\end{enumerate}

\newpage

\bibliographystyle{alpha}
\bibliography{referencesfortranslations}

\begin{thebibliography}{{Wim}27}

\bibitem[{Hil}00]{Hi1}
D.~{Hilbert}.
\newblock {Mathematische Probleme. Vortrag, gehalten auf dem internationalen
  Mathematiker-Congress zu Paris 1900.}
\newblock {\em {Nachr. Ges. Wiss. G\"ottingen, Math.-Phys. Kl.}},
  1900:253--297, 1900.

\bibitem[{Hil}26]{Hi2}
D.~{Hilbert}.
\newblock {\"Uber die Gleichung neunten Grades.}
\newblock {\em {Math. Ann.}}, 97:243--250, 1926.

\bibitem[Kle22]{Kl3}
F.~Klein.
\newblock Substitutionsgruppen und gleichungstheorie.
\newblock {\em Ges. math. Abh.}, 2:255--504, 1922.

\bibitem[{Tsc}34a]{Ts4}
N.~{Tschebotar{\"o}w}.
\newblock {Grundlagen der Galois-Theorie I.}
\newblock {Leningrad-Moskau: Staatsverlag, S. 1-221 (1934).}, 1934.

\bibitem[Tsc34b]{Ts5}
N.~Tschebotar{\"o}w.
\newblock {\"Uber das Klein-Hilbertsche Resolventenproblem.}
\newblock {\em {Bull. Soc. Phys.-Math. Kazan, III. Ser.}}, 6:5--22, 1934.

\bibitem[vdW30]{Vd6}
B.~L. van~der Waerden.
\newblock {\em Moderne Algebra}, volume~2 of {\em Grundlehren der
  mathematischen Wissenschaftern, volume 34}.
\newblock Springer-Verlag Berlin Heidelberg, 1930.

\bibitem[{Wim}27]{Wi7}
A.~{Wiman}.
\newblock {\"Uber die Anwendung der Tschirnhausen-Transformation auf die
  Reduktion algebraischer Gleichungen.}
\newblock {Nova Acta R. Soc. scient. Uppsala (4) Vol. extraord. 1927, Nr. 16, 8
  S. (1927).}, 1927.

\end{thebibliography}

\end{document}